\documentclass[a4paper,11pt,centertags]{amsart}
\usepackage[latin1]{inputenc}
\usepackage[english]{babel}
\usepackage{cite}
\usepackage{verbatim}
\usepackage{color}

\usepackage{graphicx,picture}
\usepackage{amsmath,amsthm}
\usepackage{hyperref}

\usepackage[top=2.5cm, bottom=2.5cm, left=3cm,right=3cm]{geometry}

\allowdisplaybreaks

\newcommand{\diesis}{^\#}
\newtheorem{theo}{Theorem}[section]
\newtheorem{lemma}{Lemma}[section]
\newtheorem{prop}{Proposition}[section]

\theoremstyle{definition}

\newtheorem{rem}{Remark}[section]

\numberwithin{equation}{section}

\newcommand{\R}{\mathbb R}

\newcommand{\de}{\partial}
\newcommand{\eps}{\varepsilon}
\newcommand{\ds}{\displaystyle}

\DeclareMathOperator{\esssup}{ess\,sup}

\begin{document}
\title[Constrained Cheeger type inequality]{Symmetry breaking in a constrained Cheeger type
  isoperimetric inequality} 
\author[B. Brandolini,  F. Della Pietra, C. Nitsch and
  C. Trombetti]{
Barbara Brandolini, Francesco Della Pietra,\\ Carlo Nitsch and
Cristina Trombetti}

 \address{
Barbara Brandolini \\
Universit\`a degli Studi di Napoli ``Federico II''\\
Dipartimento di Ma\-te\-ma\-ti\-ca e Applicazioni ``R. Caccioppoli''\\
Complesso Monte S. Angelo - Via Cintia\\
80126 Napoli, Italia.
}
\email{brandolini@unina.it}

\address{
Francesco Della Pietra \\
Universit\`a degli Studi di Napoli ``Federico II''\\
Dipartimento di Ma\-te\-ma\-ti\-ca e Applicazioni ``R. Caccioppoli''\\
Complesso Monte S. Angelo - Via Cintia\\
80126 Napoli, Italia.
}
\email{f.dellapietra@unina.it}

 \address{
Carlo Nitsch \\
Universit\`a degli Studi di Napoli ``Federico II''\\
Dipartimento di Ma\-te\-ma\-ti\-ca e Applicazioni ``R. Caccioppoli''\\
Complesso Monte S. Angelo - Via Cintia\\
80126 Napoli, Italia.
}
\email{c.nitsch@unina.it}

 \address{
Cristina Trombetti \\
Universit\`a degli Studi di Napoli ``Federico II''\\
Dipartimento di Ma\-te\-ma\-ti\-ca e Applicazioni ``R. Caccioppoli''\\
Complesso Monte S. Angelo - Via Cintia \\
80126 Napoli, Italia.
}
\email{cristina@unina.it}

\date{}
\maketitle
\begin{abstract}
The study of the optimal constant $\mathcal K_q(\Omega)$ in the Sobolev inequality 
\[
\|u\|_{L^q(\Omega)} \le \frac{1}{\mathcal K_q(\Omega)}\|Du \|(\R^n),
\qquad 1\le q<1^\ast, 
\]
for BV functions which are zero outside $\Omega$ and
with zero mean value inside $\Omega$,  leads to the definition of a
Cheeger type constant. We are interested in finding the best possible
embedding constant in terms of the measure of $\Omega$ alone. We set
up an optimal shape problem and we completely characterize, on varying
the exponent $q$, the behavior of optimal domains. Among other things
we establish the existence of a threshold value $1\le\tilde q< 1^\ast$
above which the symmetry of optimal domains is broken. Several
differences between the cases $n=2$ and $n\ge 3$ are emphasized.
\end{abstract}

\section{Statement of the problem and main results}
Let $\Omega$ be a bounded, open set of $\R^n$, $n\ge 2$, and let $1\le
q <1^*= \frac{n}{n-1}$.  
Denoting by $BV_0(\Omega)=\left\{u\in BV(\R^n): \, u\equiv 0 \text{ in
  } \R^n\setminus \Omega \,\right\}$ it is well-known that there
exists a constant $\mathcal C_q(\Omega)$ such that 
\begin{equation}\label{sob}
\|u\|_{L^q(\Omega)} \le \frac{1}{\mathcal C_q(\Omega)} \|Du \|(\R^n)
\end{equation}
 for all $u\in BV_0(\Omega)$. Here $||Du||(\R^n)$ denotes the total
 variation of $u$ in $\R^n$.  

\noindent The least possible constant such that \eqref{sob} holds true
is given by 
\begin{equation*}
\mathcal C_q(\Omega)=\min\left\{\frac{
    \|Du\|(\R^n)}{\|u\|_{L^q(\Omega)}}:\; u\in BV_0(\Omega),\,
  u\not\equiv 0 
 \right\},
\end{equation*}
a quantity that is well-known in literature since it coincides with
the so called Cheeger constant \cite{ch} (see also the survey paper \cite{Pa} and the references therein): 
\begin{equation*}
 \mathcal{C}_q(\Omega)= \min\left\{\frac{P(E)}{|E|^\frac1q}:
   E\subseteq \Omega,\, |E| >0 \right\}. 
 \end{equation*}
Here by $P(E)$ and $|E|$ we denote the perimeter and the Lebesgue
measure of $E$ respectively. 

\noindent An elementary scaling argument enforces
$\mathcal{C}_q(\Omega)|\Omega|^{\frac 1q -\frac{n-1}n}$ to be
invariant under dilations, therefore it is possible to optimize such a
product over all bounded, open sets $\Omega$. Indeed an interesting
consequence of the isoperimetric inequality is that 
\begin{equation}\label{cheeger_in}
\mathcal{C}_q(\Omega)|\Omega|^{\frac 1q -\frac{n-1}n}\ge n\omega_n^\frac1n,
\end{equation}
where $\omega_n$ denotes the measure of the unit ball in $\R^n$. Hence
in the class of bounded, open sets of given measure balls minimize
$\mathcal{C}_q(\Omega)$. Furthermore, given any bounded, open set, it holds that
$\Omega$ 
\begin{equation*}
\|u\|_{L^q(\Omega)} \le \frac{|\Omega|^{\frac 1q
    -\frac{n-1}n}}{n\omega_n^\frac1n}\|Du \|(\R^n) 
\end{equation*}
for all $u\in BV_0(\Omega)$.

The study of optimal constants in Sobolev--Poincar\'e inequalities for BV functions has been very popular since several decades.
Many results can be found for instance in \cite{ma}, and more recently in \cite{ci12,dpg2,dpg1,efknt}.

In this paper we consider the following minimization problem
\begin{equation}
  \label{pb}
 \mathcal{K}_q(\Omega)=\min\left\{\frac{
     \|Du\|(\R^n)}{\|u\|_{L^q(\Omega)}}:\; u\in BV_0(\Omega),\,
   u\not\equiv 0, \, \int_\Omega u\, dx=0  
 \right\},
\end{equation}
carrying the Sobolev inequality
 \[
\|u\|_{L^q(\Omega)} \le \frac{1}{{\mathcal K}_q(\Omega)} \|Du \|(\R^n)
\]
holding for functions $u\in BV_0(\Omega)$ having zero mean
value. Since in general ${\mathcal C}_q(\Omega)\le{\mathcal
  K}_q(\Omega)$, in comparison with \eqref{sob}, we are trading the restriction to zero mean value
functions for a better embedding constant. Even in this case, scaling
arguments enforce $\mathcal{K}_q(\Omega)|\Omega|^{\frac 1q
  -\frac{n-1}n}$ to be invariant under dilations and the present paper is devoted to the study of the optimal lower bound in the wake of \eqref{cheeger_in}
  and to a complete characterization of the optimal sets (from now on called ``minimizers'') on which $\mathcal{K}_q(\Omega)|\Omega|^{\frac 1q -\frac{n-1}n}$ achieves the lower bound.
  
\noindent To this aim we rewrite $\mathcal{K}_q(\Omega)$ in terms of geometric
quantities, such as perimeters and measures of subsets of
$\Omega$. For a given bounded open set $\Omega$ we define 
\begin{equation}\label{phi}
\Phi(\Omega)=\{(E_1,E_2): \; E_1,E_2 \subseteq \Omega,\; E_1\cap
E_2=\emptyset,\, |E_1|>0,\,|E_2|>0\}, 
\end{equation}
and prove the following.
\begin{theo}\label{reduction}
If $\Omega$ is a bounded open set of $\R^n$, $n\ge 2$, and
$1\le q <1^\ast$, then
\begin{equation}
  \label{set}
\mathcal{K}_q(\Omega)= \min\left\{
  \frac{\frac{P(E_1)}{|E_1|}+\frac{P(E_2)}{|E_2|} } 
 {({|E_1|}^{1-q}+{|E_2|}^{1-q})^{\frac 1 q}}
:\; (E_1,E_2) \in
\Phi(\Omega)\right\}.
\end{equation}
\end{theo}


\noindent Taking advantage of \eqref{set}, we are able to characterize the minimizers.
Astonishingly, the minimization is very sensitive to the choice of $n$
and $q$. A symmetry breaking phenomenon appears above a threshold
value of the exponent $q$. 
Furthermore for certain choices of $n$ and $q$ minimizers are not even
unique. Our main result is the following. 

\begin{theo}
 \label{maintheo} 
For all $n \ge 2$ every minimizer of $\mathcal{K}_q(\Omega)$ in the class of
bounded, open sets with given measure, is  union of two disjoint
balls. Shape and uniqueness of such minimizers depends on $n$ and
$q$. More specifically, there exists $\tilde q=\tilde q(n)\in
]1,1^\ast[$  such that, when $q<\tilde q$, the minimizer is unique and
the two balls composing the minimizer have the same radius, while, when
$q>\tilde q$, the minimizer is unique and the two balls composing the
minimizer have different radii. Moreover:
\begin{enumerate}
\item  If $n=2$, then $\tilde q=\frac{7}{4}$, the minimizer is unique
  even at $q=\tilde q$ and consists in the union of two disjoint  balls with equal
  radii. 
\item  For all $n\ge 3$,  then  
the minimizer is not unique at $q=\tilde q$, indeed there are exactly
two minimizers one of which is  the union of two disjoint balls with
equal radii. 
  \end{enumerate}
\end{theo}

\begin{rem}\label{mainrem}
We point out some additional properties that will be deduced during
the proof of the 
Theorem \ref{maintheo}. Assume that we work with the class of bounded, open sets of given measure. When $n=2$, the radii of the
balls composing the unique minimizer change continuously for
$q\in[1,2[$ and the largest one is nondecreasing with respect to $q$. 
The case $n\ge 3$ shows some differences. Bearing in mind that the
exact value $\tilde q$ is not explicitly given, we can bound it from
above and below by $1+\frac1n +\frac1{n^2}$ and $1+\frac1n$
respectively. More important, crossing the threshold value $\tilde q$
the minimizer abruptly jumps from two balls of equal radii to two balls
of unequal radii. For $q= \tilde q$ minimizing pairs of balls of equal
and unequal radii coexist. Considering the minimizing pairs of unequal
radii, for $q\ge \tilde q$, the radius of the largest ball
continuously increases with respect to $q$. \\ 
Finally, regardless the value $n\ge2$ the minimizing pair always
degenerates to one ball as $q\to 1^\ast$. 
\end{rem}



A different point of view to look at our problem consists in considering the minimization in \eqref{pb} as relaxed form of 
\begin{equation}\label{pbW}
\mathcal{K}_q(\Omega)=\inf\left\{ \frac{\ds\int_\Omega |\nabla u|
      dx}{\left(\ds\int_\Omega|u|^q dx\right)^{\frac 1 q}}:\;
    u\in W_0^{1,1}(\Omega),\; u\not\equiv 0\text{ and }\int_\Omega
    u\,dx =0\right\}. 
\end{equation}
In this case we can address to $\mathcal{K}_q(\Omega)$ as the ``twisted eigenvalue'' of the $1$-Laplacian or ``twisted Cheeger constant'',
that means $\mathcal{K}_q(\Omega)$ is the $BV$ counterpart of the so-called
twisted eigenvalue for the Laplacian:   
\begin{equation*}
\lambda^T(\Omega)=\min\left\{ \frac{\ds\int_\Omega |\nabla u|^2
      dx}{\ds\int_\Omega|u|^2 dx}:\;
    u\in H_0^1(\Omega),\; u\not\equiv 0\text{ and }\int_\Omega u\,dx
    =0\right\}. 
\end{equation*}
To our knowledge the term ``twisted eigenvalue'' was first introduced by Barbosa and B\'erard in
\cite{bb00}. Later Freitas and Henrot in \cite{fh04} employed
symmetrization arguments to show that the pairs of disjoint balls of equal
radii are the unique minimizers of $\lambda^T(\Omega)$ among all
bounded, open sets of given measure. For the interested reader,
generalization of the twisted Laplacian eigenvalue in different directions have already
been studied for instance in \cite{bcm,chp12,n12,bfnt11}. 

In the spirit of \cite{fh04}, one might expect that also minimizers for
$\mathcal{K}_q(\Omega)$ 
are pairs of disjoint balls of equal radii, and Theorem \ref{maintheo} contradicts this intuition.

The picture that we get is much more similar to the one obtained for
certain $1$-dimensional Wirtinger inequality in \cite{dgs,bkn98,n02,cd03,gn11} where the occurrence of symmetric and asymmetric minimizers have been 
completely settled. In dimension 
greater than $1$, symmetry breaking for $p$-Laplacian twisted
eigenvalue problems, for certain range of exponents, 
have been observed in an interesting remark by A. I. Nazarov recently
appeared in \cite{n12}.  

\section{Proof of the Theorem \ref{reduction}: reduction to
  characteristic functions} 
The main idea is to show that it is possible to study problem \eqref{pb} by
considering only test functions whose positive and negative parts
are characteristic functions up to multiplicative factors.  
Let us introduce the following notation. For any $u \in BV_0(\Omega)$ we set 
\[
  F_q(u)=\frac{\|Du\|(\R^n)}{\|u\|_{L^q(\Omega)}}.  
\]
Clearly, if as usual  $\chi_{E}$ denotes the characteristic function of a set
$E\subseteq \R^n$, recalling \eqref{phi},
\[
F_q\left(|E_2|\chi_{E_1}-|E_1|\chi_{E_2}\right)=\frac{\frac{P(E_1)}{|E_1|}+\frac{P(E_2)}{|E_2|} } 
 {({|E_1|}^{1-q}+{|E_2|}^{1-q})^{\frac 1 q}}, \qquad (E_1,E_2)\in \Phi(\Omega).
\] 
We denote
\[
Q(E_1,E_2)=\frac{\frac{P(E_1)}{|E_1|}+\frac{P(E_2)}{|E_2|} } 
 {({|E_1|}^{1-q}+{|E_2|}^{1-q})^{\frac 1 q}}
\]
and we define
  \begin{equation}
    \label{eq:setpb}
  \ell_q(\Omega)= \inf_{(E_1,E_2)\in  \Phi(\Omega)} Q(E_1,E_2).  
  \end{equation}
First we  prove that the infimum in \eqref{eq:setpb} is attained.
 The classical isoperimetric inequality
  implies that 
  \[
  Q(E_1,E_2) \ge n\omega_n^{\frac 1 n} \frac{{|E_1|}^{-\frac 1
      n}+{|E_2|}^{-\frac 1 n}}{({|E_1|}^{1-q}+{|E_1|}^{1-q})^{\frac 1 q}},
  \]
  and, being $1\le q<1^\ast$, the right-hand side diverges as
  $(|E_1|,|E_2|)\rightarrow (0,0)$. Hence, if $(E^k_1,E_2^k)$ is a
  sequence of couple of 
  domains in $\Phi(\Omega)$ which minimizes \eqref{eq:setpb} as
  $k\rightarrow +\infty$, the sequences $\chi_{E^k_1}$, $\chi_{E^k_2}$ are
  bounded in $BV$ and, up to subsequences, strongly converge in
  $L^1(\Omega)$. Then the lower semicontinuity of the perimeter along
  these sequences guarantees that the infimum in \eqref{eq:setpb} is attained.
  
  Now we show that
  $\mathcal{K}_q(\Omega)=\ell_q(\Omega)=\displaystyle\min_{(E_1,E_2)
    \in \Phi(\Omega)} Q(E_1,E_2)$. By choosing $\tilde
  u=|E_2|\chi_{E_1}-|E_1|\chi_{E_2}$, with $(E_1,E_2)\in
  \Phi(\Omega)$, we have that $\ds\int_\Omega 
  \tilde u\,dx=0$. Thus $\tilde u $ is an admissible test function in
  \eqref{pb}, and $F_q(\tilde u)= Q(E_1,E_2)$. Hence, we have that
  \[
  \mathcal{K}_q(\Omega) \le \ell_q(\Omega).
  \]
  In order to conclude the proof, we have to show that, for any
  admissible function $u \in BV_0(\Omega)$,
  \[
  F_q(u) \ge \ell_q(\Omega).  
  \]
  Using standard notation, let $u=u_+-u_-$, where $u_+$ and $u_-$
  are, respectively, the positive and 
  negative part of $u$, and set 
  $\Omega_\pm={\rm spt\, } u_\pm$.

\noindent Clearly  $\ds\int_{\Omega_+}
  u_+ dx=\ds\int_{\Omega_-}u_- dx$. Moreover being $u\not\equiv 0$
  implies that $|\Omega_\pm|>0$. Let
  \[
  \mu_+(t)= |\{u_+>t\}|,\quad \mu_-(t) =|\{u_->t\}|.
  \]
 By the Fubini Theorem and the H\"older inequality, we get that
  \begin{multline*}
    \int_{\Omega^+} u_+^q dx =
    \int_\Omega \left( u_+(x)^{q-1} \int_0^{+\infty}
      \chi_{\{u_+>t\}}(x)dt\right)dx= \\
   = \int_0^{+\infty} \left( \int_{\{u_+>t\}}  u_+^{q-1}
  dx \right) dt \le \left(
  \int_{\Omega_+}  u_+^{q} dx \right)^{1-\frac 1 q}\int_0^{+\infty} \
\mu_+(t)^{\frac 1 q} dt,
\end{multline*}
and the same relation holds for $u_-$. Using the above
inequality we deduce that
\begin{equation}
  \label{eq:1}
\left(\int_\Omega {|u|}^q dx\right)^{\frac 1 q} \le \left[\left(
  \int_0^{\esssup u_+} \mu_+(t)^{\frac 1 q} dt \right)^q+ \left(
  \int_0^{\esssup u_-} \mu_-(s)^{\frac 1 q} ds \right)^q \right]^{\frac 1
q}.
\end{equation}
Now we perform the change of variables
\[
\xi(t) = \int_0^t \mu_+(\sigma) d\sigma,\quad t \in [0,\esssup u^+],
\]
and
\[
\eta(s) = \int_0^s \mu_-(\tau) d\tau,\quad s \in[0,\esssup u^-],
\]
respectively, in both integrals in the right-hand side of
\eqref{eq:1}. The functions $\xi$ and $\eta$ are strictly increasing and,
being $\int_{\Omega^+} u_+\, dx=\int_{\Omega^-} u_-\, dx:=M$, we 
have that $\xi(t)\le M=\xi(\esssup u^+)$ and $\eta(s)\le 
M=\eta(\esssup u^-)$. Hence, from 
\eqref{eq:1} and the Minkowski inequality (see for example
\cite[Theorem 202, p. 148]{hlp}) it follows that 
\begin{multline}
  \label{eq:2}
\left(\int_\Omega {|u|}^q dx\right)^{\frac 1 q} \le \left[\left(
  \int_0^{M} \mu_+(t(\xi))^{\frac{1-q}{q}} d\xi \right)^q+ \left(
  \int_0^{M} \mu_-(s(\eta))^{\frac{1-q}{q}} d\eta \right)^q
\right]^{\frac 1 q}
\le\\
\le 
  \int_0^{M}
  \left[
    \mu_+(t(r))^{{1-q}} + \mu_-(s(r))^{{1-q}}\right]^{\frac 1 q}
  dr.
\end{multline}
On the other hand,
denoting by $p_\pm(t)=P(\{u_\pm>t\})$ for $t\ge0$, 
the co-area formula for BV functions yields 
\begin{multline}
  \label{eq:3}
  \|Du\|(\R^n)=\|Du_+\|(\R^n)+\|Du_-\|(\R^n)
 \\= \int_0^{+\infty} p_+(t)\,dt + \int_0^{+\infty} p_-(s)\, ds = 
  \int_0^M \left[
    \frac {p_+(t(r))}{\mu_+(t(r))} + \frac{p_-(s(r))}{\mu_-(s(r))}
    \right]dr,
\end{multline}
where we performed the change of variables $t=t(\xi)$, $s=s(\eta)$ defined
above. Finally, combining \eqref{eq:2} and \eqref{eq:3} we have
\begin{multline*}
F_q(u)\ge \frac{\int_0^M \big[
    \frac {p_+(t(r))}{\mu_+(t(r))} + \frac{p_-(s(r))}{\mu_-(s(r))}
    \big]dr}{ \int_0^{M}
  \big[
    \mu_+(t(r))^{{1-q}} + \mu_-(s(r))^{{1-q}}\big]^{\frac 1 q}
    dr } \ge
    \\[.2cm] \ge \inf_{0<r<M} \frac{  \frac {p_+(t(r))}{\mu_+(t(r))} +
    \frac{p_-(s(r))}{\mu_-(s(r))} }{\big[ \mu_+(t(r))^{{1-q}} +
    \mu_-(s(r))^{{1-q}}\big]^{\frac 1 q} } 
  \ge
  \inf_{(E_1,E_2)\in\Phi(\Omega)} Q(E_1,E_2)=\ell_q(\Omega),
\end{multline*}
and this concludes the proof.

\section{Proof of Theorem \ref{maintheo}}
We split the proof into two subsections. In the first one we prove that the minimum of $\mathcal{K}_q(\Omega)$ among
bounded, open sets of given measure is attained at the union of two disjoint
balls. In the second subsection we characterize the minimizers.
\subsection{An isoperimetric inequality for $\mathcal{K}_q(\Omega)$}
 We denote by $\mathcal B(t)$ the family of sets with measure $t$
 which are union of two disjoint balls.  
\begin{prop}
  \label{bound}
Let $\Omega$ be a
 bounded, open set of $\R^n$, with $\Omega\not\in\mathcal
 B(|\Omega|)$ and $1\le q <1^\ast$. 
There  exists a set $ \tilde\Omega=\tilde B_1 \cup \tilde B_2\in
\mathcal B(|\Omega|)$,  
 such that 
 \begin{equation}\label{1}
  \mathcal{K}_q(\Omega)> \mathcal{K}_q(\tilde\Omega)=\min_{A\in
    \mathcal B(|\Omega|)} 
  \mathcal{K}_q(A). 
\end{equation}
  Moreover
\begin{equation}\label{2}
  \mathcal{K}_q(\tilde \Omega)=Q(\tilde B_1,\tilde B_2).
\end{equation}
\end{prop}
\begin{proof}
Let $u$ be a minimizer for \eqref{pb}. By \eqref{set} there
exists a couple $(E_1, E_2)\in \Phi(\Omega)$ such that
\[
\mathcal{K}_q(\Omega)=Q(E_1,E_2).
\]
The standard isoperimetric inequality implies that
\begin{equation}
  \label{eq:4}
\mathcal{K}_q(\Omega)=Q(E_1,E_2) \ge Q(E_1\diesis,E_2\diesis)\ge
\min_{(F_1,F_2)\in 
  \Phi( A) } Q(F_1,F_2)=\mathcal{K}_q( A),
\end{equation}
where $E_1\diesis,E_2\diesis$ are two disjoint balls such that
$|E_i\diesis|=|E_i|$, $i=1,2$, $ A=B_1\cup B_2\in \mathcal
B(|\Omega|)$, and 
$E_i\diesis\subseteq B_i$, for $i=1,2$. Then there exists $\tilde
\Omega=\tilde B_1 \cup \tilde B_2 \in  \mathcal 
B(|\Omega|)$ such that
\begin{equation}
  \label{eq:5}
  \mathcal{K}_q(\Omega)\ge \mathcal{K}_q( A)\ge \min_{A\in \mathcal
    B(|\Omega|)} 
  \mathcal{K}_q(A) =\mathcal{K}_q(\tilde \Omega).
\end{equation}
Clearly, the first inequality in \eqref{eq:4} holds as an equality if
 and only if $E_1,E_2$ are balls. In this case, since $\Omega\not\in
 \mathcal B(|\Omega|)$, then $|E_1|+|E_2|<|\Omega|$.
As matter of fact, it is easy to see that, being $q<1^\ast$,
$Q(E_1,E_2)$ is strictly decreasing with respect to homotheties of
$E_1\cup E_2$. This implies that the second inequality in
 \eqref{eq:5} is strict, and the proof of \eqref{1} is completed. 
 
 \noindent Now suppose that $\mathcal{K}_q(\tilde \Omega)=Q(\tilde
 E_1, \tilde E_2)$. The couple $(\tilde E_1,\tilde E_2)\in \Phi(\tilde
 \Omega)$ is such 
  that each $\tilde E_i$ is contained just in one ball. Indeed, if for
  example $\tilde E_1\cap \tilde B_1=\tilde E^a_{1}$, $\tilde E_1\cap
  \tilde B_2=\tilde E^b_{1}$, both with 
  positive measure, and $\frac{P(\tilde E^a_1)}{|\tilde E^a_1|}\le
  \frac{P(\tilde E^b_1)}{|\tilde E^b_1|}$, we have that 
  \[
  \frac{P(\tilde E_1)}{| \tilde E_1|} =
  \frac{P(\tilde E_1^a)+P(\tilde E_1^b)}{|\tilde E_1^a|+|\tilde
    E_1^b|}\ge 
  \frac{P(\tilde E_1^a)}{|\tilde E_1^a|},
  \]
  which implies, being $|\tilde E_1|>\max\{|\tilde E_1^a|,|\tilde
  E_1^b|\}$, that  
 $ Q(\tilde E_1,\tilde E_2)>Q(\tilde E_1^a, \tilde E_2)$ contradicting
 the minimality of $(\tilde E_1, \tilde E_2).$ 
 Now suppose that $\tilde E_1\cup \tilde E_2\neq \tilde B_1\cup \tilde
 B_2$.  
 Then $|\tilde E_1|+|\tilde E_2|<|\Omega|$. Moreover, by the standard
 isoperimetric inequality, 
 \[
  Q(\tilde E_1,\tilde E_2) \ge Q(\tilde E_1\diesis, \tilde
  E_2\diesis). 
 \]
 Finally, being $Q$ strictly monotone with respect to the
 homotheties,  there exist two balls $F_1$ and $F_2$ such that
 $|F_1|+|F_2|=|\Omega|$ and 
 \[
 Q(\tilde E_1\diesis, \tilde E_2\diesis) > Q(F_1,F_2),
 \]
contradicting the minimality of $\tilde \Omega = \tilde B_1 \cup
\tilde B_2$. 
\end{proof}

\subsection{Properties of the minimizers and symmetry breaking}

Since the quantity $\mathcal{K}_q(\Omega)|\Omega|^{\frac 1q
  -\frac{n-1}n}$ is invariant under dilations, we are free to
arbitrarily fix the measure of $\Omega$. Indeed the shape of
minimizers of $\mathcal{K}_q(\Omega)$ under measure constraint is not
affected by the measure chosen.  

For simplicity we will restrict our analysis to the case
$|\Omega|=\omega_n$ (the measure of the unit ball in $\R^n$). Our
minimization problem is 
\[
\min_{|\Omega|=\omega_n} \mathcal{K}_q(\Omega) 
\]
and Proposition \ref{bound} implies that the study
 reduces to the study of a one-dimensional minimum problem. More
 precisely we have to minimize 
\[
 Q(B_1,B_2)=n\omega_n^{1-\frac 1 q}
\frac{ r_1^{-1}+r_2^{-1} }
{ \left( r_1^{-n(q-1)}+r_2^{-n(q-1)} \right)^{\frac 1 q} }, \]
over all possible pairs of balls $B_1$ and $B_2$ of radii $r_1$ and
$r_2$ respectively, under the restriction $r_1^n+r_2^n=1$.

\noindent We introduce the new
variable \[x=\frac12\log\left(\frac{r_1^n}{1-r_1^n}\right).\] The
value $x=0$ corresponds to $r_1=r_2$ and $x$ is a monotone increasing
function of $r_1$ from $-\infty$ to $+\infty$, as $r_1$ goes from $0$
to $1$.    
We have,
\[
\min_{|\Omega|=\omega_n}\mathcal{K}_q(\Omega) = n2^{\frac 1
  n}\omega_n^{1-\frac 1 q} \min_{x\in \R} f_n(x,q), 
\]
where
\begin{equation}
\label{fq}
f_n(x,q)=\Big[\cosh x\Big]^{\frac 1 q + \frac 1 n - 1 }
\Big[\cosh \left(\frac x n\right)\Big]\Big[\cosh
(x(q-1))\Big]^{-\frac 1 q}. 
\end{equation}

\noindent For every $q\in[1,1^\ast[$, $f_n\left(0,q\right)=1$ and
obviously $f_n(x,q)$ is symmetric about $x=0$. Therefore we also have
$\partial_xf_n$ vanishing at $x=0$, in fact two balls with equal radii
are always a stationary point of the functional
$Q(B_1,B_2)$. Moreover, $f_n(x,q)$ diverges for $|x|\rightarrow
\infty$. The behavior of $f_n(\cdot,q)$ is very sensitive to the values of $n$ and $q$ as shown in Figures \ref{fig1}-\ref{fig2}.

%

All the statements in Theorem \ref{maintheo} are consequences of
several claims we are going to prove. 

  \medskip

{\bf Claim 1.} {\em For any given $n\ge 2$ and any given
  $q\in[1,1^\ast[$, the function $f_n(\cdot,q)$ has at most two local
  minimum points in $[0, \infty[$, and not more than one in
  $]0,\infty[$.} 

 
 \medskip
 
\noindent {\it Proof of Claim 1}. Differentiating $f_n(x,q)$ with
respect to $x$, we have  
\begin{multline*}
\de_x\, f_n(x,q) = c_n(x,q) \left\{  \left( \frac 1 n + \frac 1 q
-1\right)\sinh\left[x\left(q+\frac 1 n\right)\right] + \frac 1 n 
\sinh\left[x\left(2-q+\frac 1 n\right)\right]\right. + \\ \left.-\left(1-
 \frac 1 q \right)\sinh\left[x\left(q-\frac 1 n\right)\right]\right\}
=c_n(x,q)  A_n(x,q),
\end{multline*}
where
\[
c_n(x,q)={\frac12}\Big[\cosh(x(q-1))\Big]^{-\frac1q-1}
\Big[\cosh x \Big]^{\frac1q+\frac1n-2}>0
\]
and $A_n(x,q)$ is the function in the
braces. $A_n(0,q)=0$ and the claim is proved if we show
that $A_n(\cdot,q)$ has no more than two zeros in $]0,+\infty[$. This
is an immediate consequence of the following. 

\begin{lemma}
  \label{lemmasinh}
 Any nontrivial linear combination of three hyperbolic sinus functions
 has at most two positive zeros. 
\end{lemma}
\begin{proof}
  Let $A(x)=a\sinh(\alpha x)+b\sinh(\beta x)+c\sinh(\gamma x)$, with
  $a,b$ and $c$ nonzero real numbers, and $\alpha,\beta,\gamma \ge 0$,
  such that $A\not\equiv 0$. Clearly, if $a,b$ and $c$ have the same
  sign, the claim of the lemma is obvious. Hence, without loss of
  generality we can consider linear combinations as follows: 
  \[
  \begin{array}{llccccc}
  A(x)&=&a \sinh\big(\sqrt{1+\eps}\, x\big) &+&
  b\sinh\big(\sqrt{1+\delta}\,
  x\big)&-&c \sinh x\\[.1cm]
  &=&X(x)&+&Y(x)&-&Z(x).
  \end{array}
 \]
 with $a,b,c> 0$, $\eps > -1$, $\eps \ne 0$ and $\delta \ge
 -1$. Obviously, $A(0)=0$. Moreover,
 \begin{equation}
   \label{sh1}
  A''(x)=A(x)+\eps X(x)+\delta Y(x).
 \end{equation}
Suppose that there exists a nonpositive minimum point  $x_0>0$ of
$A$. Then,
 \[
 A(x_0)\le 0,\quad A''(x_0)\ge 0
 \]
 and, together with \eqref{sh1}, 
 \[
 \eps X(x_0)+\delta Y(x_0) \ge 0.
 \]
 Moreover, the function $\eps X+\delta Y$ cannot have a nonnegative
 maximum in $]0,+\infty[$ being, for any $x> 0$,
 \begin{equation*}
\big[ \eps X(x)+\delta Y(x) \big]''
=\eps(1+\eps)X(x)+\delta(1+\delta)Y(x) > \eps X(x)+ \delta Y(x).
\end{equation*}
Hence,
\begin{equation}
  \label{sh2}
 \eps X(x)+\delta Y(x) > 0\quad\text{for any }x> x_0.
\end{equation}
Finally, \eqref{sh1} and \eqref{sh2} imply that
 \begin{equation*}
 A''(x) > A(x) \quad \text{for any }x> x_0,
 \end{equation*}
and then the function $A(x)$ admits at most one zero in
$]x_0,+\infty[$. 
\end{proof}


\medskip

{\bf Claim 2.} {\em For any given $n\ge 2$, and any $x\in]0,\infty[$, the functions $f_n(x,\cdot)$ and $\partial_x f_n(x,\cdot)$ are decreasing in $[1,1^\ast[$.}

\medskip

\noindent{\it Proof of Claim 2}. Let us compute
\begin{multline*}
  \de_q \log( f_n(x,q) )=\frac{\de}{\de q}
  \left(
    \left(\frac 1 q + \frac 1 n-1\right) \log(\cosh x)-\frac 1 q
    \log(\cosh(x(q-1)) \right) =\\ 
  =
  -\frac{1}{q^2}\big[\log(\cosh x ) - \log(\cosh (x(q-1)))\big] -
  \frac 1 q x \tanh(x(q-1)) <0.  
\end{multline*}
For any fixed $q$, such derivative is decreasing with respect
to $x$ in  $[0,+\infty[$. Indeed, for $x>0$ we have
\begin{equation*}
\de_{q\,x}\,\log( f_n(x,q)) = -\frac{1}{q^2}\tanh x -
 \frac{1}{q^2} \tanh(x(q-1)) - x \left(1-\frac 1 q\right)
 \left[\cosh(x(q-1))\right]^{-2}< 0. 
\end{equation*}

\begin{rem}\label{rem_mon}
Claim 2 has many implications. Among all, we deduce that if $x_0>0$ is
a global minimum point for the function $f_n(\cdot,q)$, for certain
$q=q_0\in[1,1^\ast[$, then for all $q\in ]q_0,1^\ast[$ global minimum
points cannot exist in $[0, x_0[$. Roughly speaking positive global
minimum points move left to the right.  
\end{rem}

\medskip

{\bf Claim 3.} {\em If $1\le q \le \frac 7 4$, the function $f_2(\cdot, q)$ is increasing in $[0,\infty[$.} 

\smallskip

\noindent{\it Proof of Claim 3.} In view of Claim 1 and Claim 2 it is
enough to observe that the function $A_2(\cdot,\frac 74)$ is
positive in $]0,\infty[$. Indeed 
\begin{multline*}
A_2\left(x,\frac 74\right)=\frac1{14}\left\{
  \sinh\left(\frac94x\right)+
  7\sinh\left(\frac34x\right) - 6\sinh\left(\frac54x\right)\right\}\\
=\frac2{7} \sinh^3\left(\frac 14x\right) \left\{ 6\cosh
  x+2\cosh\left(\frac32 x\right)-1 \right\} 
> 2 \sinh^3\left(\frac 14x\right) 
\end{multline*}
for all $x>0$.

\medskip

{\bf Claim 4.} {\em If $\frac 7 4< q < 2$, the function $f_2(\cdot,
  q)$ has a unique local (and global) minimum point $\bar x(q)$ in $[0,\infty[$
  which is not $x=0$.} 

\smallskip

\noindent{\it Proof of Claim 4.} We observe that 

\begin{equation}\label{eq_der2}
\partial_{xx}f_n(0,q)=c_n(0,q)\,\de_x
A_n(0,q)=\left(-q+1+\frac1n+\frac{1}{n^2}\right),
\end{equation}
hence $\partial_{xx}f_2(0,q)<0$ if $q>\frac74$. Hence $x=0$ is not a
local minimum point for $\frac74<q<1^\ast$ and we complete the proof
using Claim 1. 

\begin{rem}\label{rem_case2}
Denoting by $\bar x(q)$ the unique minimum point of $f_2(\cdot,q)$,  $\bar x(q):[1,2[\to[0,\infty[$ is the continuous nondecreasing function represented in Figure \ref{fig3}. \end{rem}


\medskip

{\bf Claim 5.} {\em For any $n\ge 3$ and $1\le q \le 1+\frac 1n $, the
  function $f_n(\cdot, q)$ is increasing in $[0,\infty[$.} 

\smallskip

\noindent{\it Proof of Claim 5.} In view of Claim 2 it is enough to
prove the monotonicity in $[0,\infty[$ 
just for $q=1+\frac 1 n$. In fact we have that
\begin{equation*}
f_n\big(x,1+\textstyle \frac 1 n\big)= 
  \big[\cosh x \big]^\frac{1}{n^2+n} 
  \left[\cosh\left(\dfrac xn\right)\right]^\frac{1}{n+1}.
\end{equation*}


\medskip

{\bf Claim 6.} {\em For any $n\ge 3$ and $1+\frac{1}{n} +\frac{1}{n^2}
  \le q < 1^\ast$, the function $f_n(\cdot, q)$ has a local maximum
  point at $0$ and therefore a unique local (and global) minimum point
  in $]0,\infty[$.} 

\medskip

\noindent {\it Proof of Claim 6.} In view of Claim 2 it is enough to
prove the claim for $q=1+\frac 1 n +\frac{1}{n^2}$. Observe also that
for all $1+\frac 1 n +\frac{1}{n^2}<q<1^\ast$ the statement of Claim 6
is a trivial consequence of \eqref{eq_der2}. Let us therefore consider
$q=\bar q=1+\frac 1 n +\frac{1}{n^2}$. We have that
\[
 \de_x\, f_n(x,\bar q)= c_n(x,\bar q) A_n(x,{\bar q}),
  \]
  with
\begin{multline*}
   A_n(x,{\bar q})= \frac{1}{n(n^2+n+1)}
   \left\{
     \sinh\left(  \frac{ (n+1) ^{2}}{n^2} x \right)
     +
( {n}^{2}+n+1)
 \sinh\left( \frac{{n}^{2}-1
     }{{n}^{2}} x\right)+\right.\\ \left.-
 n(n+1) \sinh \left( \frac{ {n}^{2}+1
     }{{n}^{2}}x\right)\right\}.
 \end{multline*}
 A straightforward computation shows that
 \[
 A_n(0,{\bar q})= \partial_xA_n(0,{\bar q})=\partial_{xx}A_n(0,{\bar
   q})=0, 
 \]
 and
 \[
 \partial_{xxx}A_n(0,{\bar q})=
 4\,\frac{-n^5+3n^3+5n^2+4n+1}{n^6(n^2+n+1)}. 
 \]
 For $n=m+3$, the polynomial in the numerator becomes
 \[
 -m^5-15m^4-87m^3-238m^2-290m-104,
 \]
 which is negative if $m=n-3\ge 0$.
 Then, in this case we have
 \[
 \frac{\de f_n}{\de x} (0,{\bar q})= \frac{\de^2 f_n}{\de x^2} (0,{\bar q})= \frac{\de^3 f_n}{\de x^3}
 (0,{\bar q})=0,\quad\text{and}\quad 
  \frac{\de^4 f_n}{\de x^4} (0,{\bar q})<0,
 \]
 which proves that $x=0$ is a local maximum point for $f_n(x,\bar q)$.


 \medskip

{\bf Claim 7.} {\em For any $n\ge 3$, there exists a $\tilde q\in
  \left]1+\frac{1}{n},1+\frac{1}{n} +\frac{1}{n^2}\right[$ such that
  the function $f_n(\cdot, \tilde q)$ has two global minimum points in
  $[0,\infty[$, one of which at $0$. As a consequence the function
  $f_n(\cdot, q)$ has a unique global minimum point in $[0,\infty[$
  for all $q\in[1,1^\ast[-\{\tilde q\}$. In particular let $\bar x(q)$
  be such a unique minimum point, then $\bar x(q)=0$ for
  $q\in[1,\tilde q[$ and $\displaystyle\lim_{q\to\tilde q^+}\bar
  x(q)>0$ (see Figure \ref{fig4}).} 

\smallskip

\noindent {\it Proof of Claim 7.} We use a continuity argument by
taking advantage of the smoothness 
of $f_n$ in $[0,\infty[\times[1,1^\ast[$. All we have to prove is the
existence of a value $\tilde q$  such that $f_n(\cdot,\tilde q)$ has
more than one global minimum point. Claim 1 and the monotonicity
properties stated in Claim 2 immediately imply the rest of the
statement. 

For any given $n\ge 3$, let $\tilde q$ be the supremum of all
$q\in[1,1^\ast[$ such that the function $f_n(\cdot, q)$ achieves a
global minimum at $x=0$. In view of Claim 5 and Claim 6 we know that
$\tilde q\in \left]1+\frac{1}{n},1+\frac{1}{n}
  +\frac{1}{n^2}\right[$. Since from \eqref{eq_der2} we have
$\partial_{xx}f_n(0,\tilde q)>0$, then by definition of $\tilde q$
there exists $\tilde x>0$ such that $f_n(\tilde x,\tilde
q)=f_n(0,\tilde q)$. Both $x=0$ and $x=\tilde x$ are global minimum
points for $f_n(\cdot,\tilde q)$ in $[0,\infty[$ and the claim is
proved. With Claim 7 the proof of Theorem \ref{maintheo} is complete.

%

\medskip

Concerning Remark \ref{mainrem}, the fact that the radius of the
largest ball in the minimizers increases with $q$ is a consequence of
Claim 2 (see Remark \ref{rem_mon}). Moreover the minimizers degenerate
to one ball as $q\to1^\ast$ since the minimum point of the function
$f_n(x,q)$ diverges as $q\to 1^\ast$. This is a consequence of the
fact that $f_n(x,q)$ converges as $q \to 1^\ast$ (monotonically with
respect to $q$) to 
\[
f_n^\ast(x)=
\Big[\cosh \left(\frac x n\right)\Big]\Big[\cosh  \left(\frac
  x{n-1}\right)\Big]^{-1+\frac 1n},
\]  
and $f_n^\ast(x)$ is in fact decreasing in $[0,\infty[$ in view of
\[
\partial_x f_n^\ast(x)=-\frac{1}{n}\Big[\sinh \left(\frac x
  {n^2-n}\right)\Big]\Big[\cosh  \left(\frac
  x{n-1}\right)\Big]^{-2+\frac 1n}.
\]

\newpage

\begin{figure}[h]
\centering
\def\svgwidth{.95\columnwidth}
  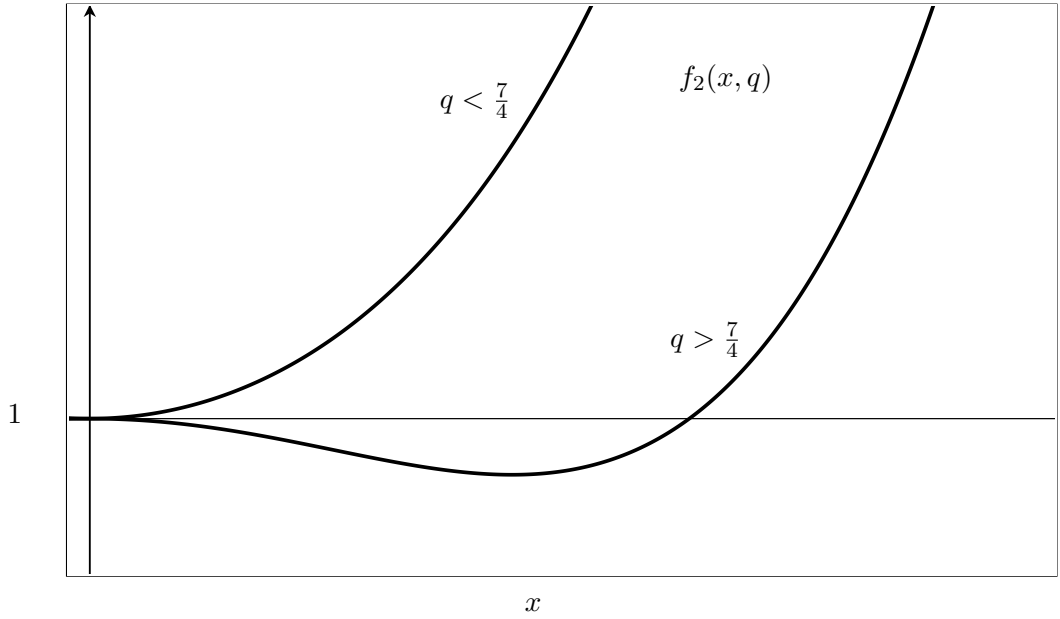\caption{Varying the value of $q$ changes the
    shape of the function $f_2(x,q)$. For $q\le \frac74$ there is only
    one stationary (global minimum) point in the origin, while for $q>
    \frac74$ another stationary point appears, the origin becomes a
    local maximum point and the minimum point shifts on positive
    values.} \label{fig1}
\end{figure}

\begin{figure}[h]
\centering
\def\svgwidth{0.95\columnwidth}
  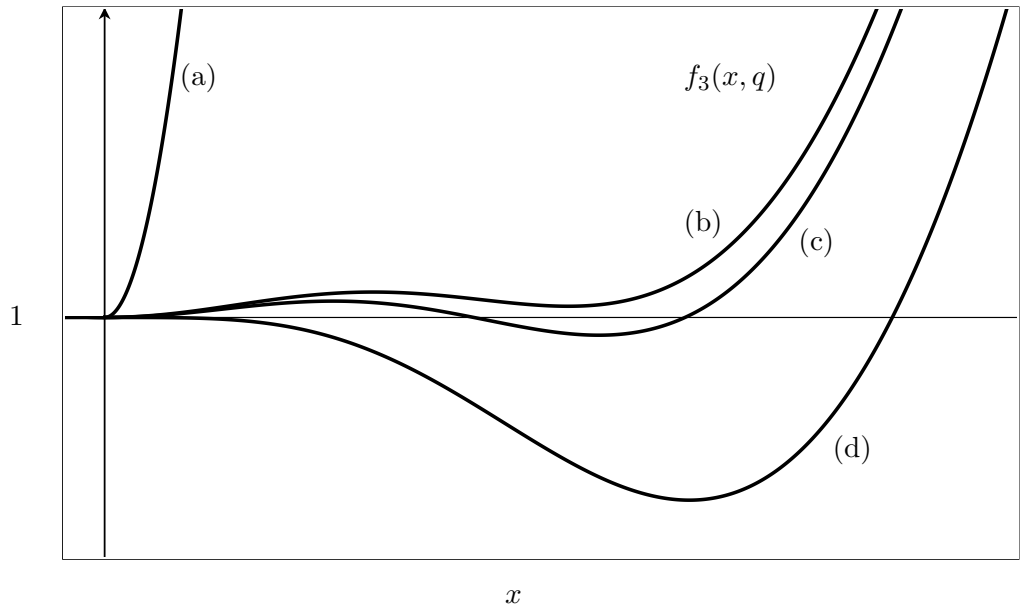\caption{In the picture, we present four different
    behaviors of the function $f_3(x,q)$, corresponding to increasing
    values of $q$, from (a) to (d). In (a) there is only one
    stationary point: a global minimum point at 0. In (b) there are
    three stationary points and the global minimum point is at 0. In
    (c) there are three stationary points and the origin is still a
    local minimum point yet not a global one. In (d) there are two
    stationary points and the origin becomes a local maximum
    point. The picture is about the same for all $n\ge 3$.} \label{fig2}
\end{figure}

\begin{figure}[h]
\centering
\def\svgwidth{1.0\columnwidth}
  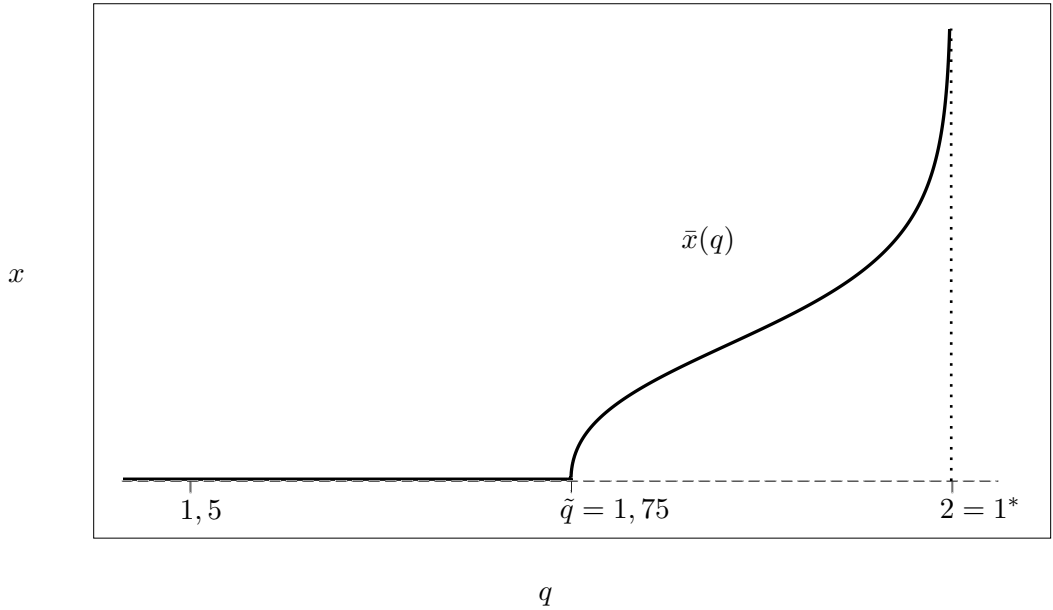\caption{The graph, for $n=2$, of $\bar x(q)$
    (defined in Claim 4). The distance from the origin of the global
    minimum point is  nondecreasing and continuous with respect to
    $q$.} \label{fig3}
\end{figure} 

\begin{figure}[h]
\centering
\def\svgwidth{.95\columnwidth}
  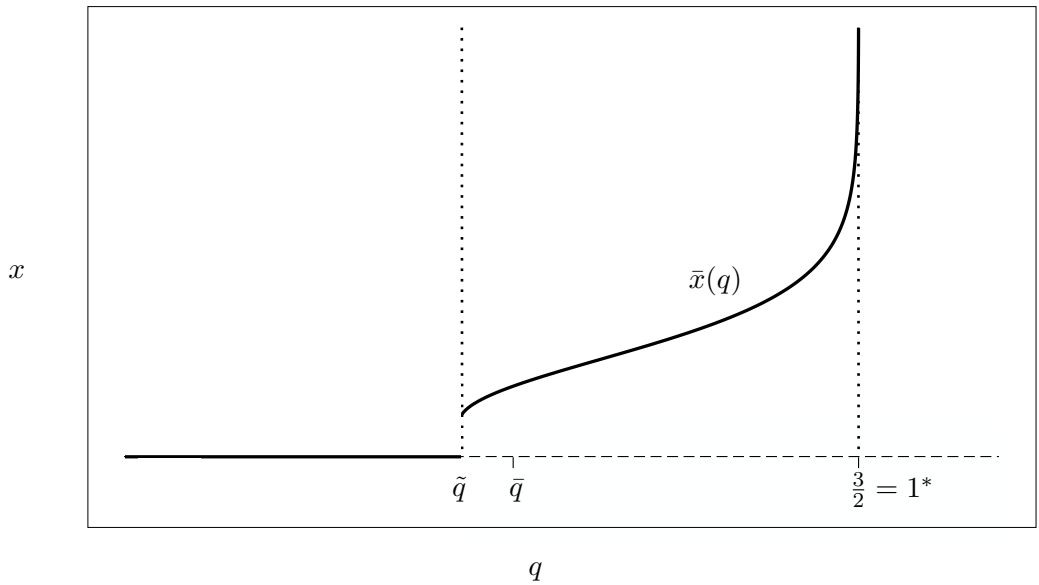\caption{The graph, for $n=3$, of $\bar x(q)$
    (defined in Claim 7). The distance from the origin of the global
    minimum point is  nondecreasing and discontinuous at $q=\tilde
    q$. Here $\bar q=\frac{13}{9}$. The behavior is similar for $n>3$
    with $\bar q=1+\frac1n+\frac1{n^2}$.} \label{fig4}
\end{figure}

\end{document}